\documentclass[11pt]{article}
\usepackage{ amssymb }
\usepackage{lipsum}

\usepackage{amsfonts}
\usepackage{amsthm}

\usepackage{epsfig}
\usepackage{graphicx}  
  
\usepackage{pdfsync}

\usepackage[utf8]{inputenc}

\usepackage{caption}
\usepackage{subcaption}
 
\usepackage{amsmath,amsfonts,amssymb,amsthm,epsfig,epstopdf,titling,url,array}

\usepackage{hyperref}

\theoremstyle{plain}
\newtheorem{thm}{Theorem}[section]
\newtheorem{lem}[thm]{Lemma}

\newtheorem{cor}{Corollary}

\theoremstyle{definition}
\newtheorem{defn}{Definition}[section]
\newtheorem{conj}{Conjecture}[section]
\newtheorem{exmp}{Example}[section]

\theoremstyle{remark}

\theoremstyle{remark}

\usepackage{graphicx}
\usepackage{xcolor}
\definecolor{darkblue}{rgb}{0,0,0.5} 
\usepackage{transparent}

\begin{document}

\title{Coarsely separation of groups and spaces.}
%\author{Panagiotis Tselekidis \\ panagiotis.tselekidis@maths.ox.ac.uk}

\author{Panagiotis Tselekidis}

%\author{Panagiotis Tselekidis \\ email: \href{mailto:panagiotis.tselekidis@maths.ox.ac.uk}{panagiotis.tselekidis@maths.ox.ac.uk}} 

\maketitle

%\href{mailto:panagiotis.tselekidis@maths.ox.ac.uk}{panagiotis.tselekidis@maths.ox.ac.uk}
%\email{\textit{email:} panagiotis.tselekidis@maths.ox.ac.uk}

%\address{Mathematical Institute, University of Oxford, 24-29 St Giles',
%Oxford, OX1 3LB, U.K.}

\begin{abstract} 
Inspired by a classical theorem of topological dimension theory, we prove that every geodesic metric space of asymptotic dimension $n$ containing a bi-infinite geodesic can be coarsely separated by a subset $S$ of asymptotic dimension equal to or smaller than $n-1$.\\
We define asymptotic Cantor manifolds, and we prove that every finitely generated group contains such a manifold.
We also state some questions related to them.
\end{abstract}

\tableofcontents

\section{Introduction.}
In 1993, M. Gromov introduced the notion of the asymptotic dimension of metric spaces (see \cite{Gr}) as an invariant of 
finitely generated groups. It can be shown that if two metric spaces are quasi isometric then they have the same asymptotic dimension.

In 1998, the asymptotic dimension achieved particular prominence in geometric group theory after a paper of Guoliang Yu, (see \cite{Yu}) which proved the Novikov higher signature conjecture for manifolds whose fundamental group
has finite asymptotic dimension. Unfortunately, not all finitely presented groups have finite asymptotic dimension. For example, Thompson's group $F$ has infinite
asymptotic dimension since it contains $\mathbb{Z}^{n}$ for all $n$.\\
However, we know for many classes of groups that they have finite asymptotic dimension, for instance, hyperbolic, relatively hyperbolic, Mapping Class Groups of surfaces, Coxeter groups, and one relator groups have finite asymptotic dimension (see \cite{BD08}, \cite{Os}, \cite{BBF}, \cite{Mats}, \cite{PT},\cite{TJ}, \cite{PTA}, \cite{PTR}).\\

The asymptotic dimension $asdimX$ of a metric space $X$ is defined as follows: $asdimX \leq n$ if and only if for every $R > 0$ there exists a uniformly bounded covering $\mathcal{U}$ of $X$ such that the R-multiplicity of $\mathcal{U}$ is smaller than or equal to $n+1$ (i.e. every R-ball in $X$ intersects at most $n+1$ elements of $\mathcal{U}$).\\

The asymptotic dimension theory bears a great deal of resemblance to dimension
theory in topology, so it is natural to ask whether results of topological dimension theory have an asymptotic analogue.

It is known that every topological space $X$ of dimension $n \geq 1$ can be separated by a subset $Y$ of dimension equal to or smaller than $n-1$ (see \cite{HuWa}). By separation we mean that the space $X \setminus Y$ contains at least two connected components. One may ask whether there exists an asymptotic analogue of that for metric spaces.

The aim of this paper is to show that for geodesic metric spaces containing a bi-infinite geodesic an asymptotic analogue exists, and to prove, for groups, an asymptotic version of Aleksandrov's theorem (\ref{Coarse.4}).\\

Let $X$ be a metric space and $K \subseteq X$. 
%We say that $K$ is \textit{deep} subset of $X$ if $X$ does not contained in $N_{R}(K)$ for every $R > 0$.\\
We say that a component $C$ of $X \setminus N_{R}(K)$ is \textit{deep} if $C$ is not contained in $N_{R^{\prime}}(K)$ for any $R^{\prime} > 0$. We recall that by $N_R(A)$ we denote the $R$-neighbourhood of a set $A$ (i.e. all the points of distance at most $R$ from $A$).\\
We say that $K$ \textit{weakly coarsely separates} $X$ if there is an $R > 0$ such that $X \setminus N_{R}(K)$ contains at least two deep components of $X$ and at least one of them is connected.\\
We say that $K$ \textit{coarsely separates} $X$ if there is an $R > 0$ such that $X \setminus N_{R}(K)$ contains at least two connected deep components of $X$.
We say that a finitely generated group $G$ can be (weakly) coarsely separated by a set $Y$ if there exists a finite generating set $S$ such that $Y \subseteq Cay(G,S)$ and $Y$ (weakly) coarsely separates the Cayley graph.

\begin{flushleft}
\textbf{Convension:} In what follows we consider only metric spaces with finite asymptotic dimension.
\end{flushleft}

We show that:

\begin{thm}\label{Coarse.1}
Let $X$ be a geodesic metric space containing a bi-infinite geodesic. If $asdimX=n>0$, then there exists a subspace $Y$ of asymptotic dimension strictly less than $n$ which coarsely separates $X$. 
\end{thm}

Firstly, we show that: 
\begin{thm}\label{Coarse.2}
Every finitely generated group $G$ can be weakly coarsely separated by a set $Y$ of asymptotic dimension strictly less than $asdimG$ (provided that $asdimG>0$).
\end{thm}

The proof of this theorem can be used to show that 
\begin{thm}\label{Coarse.3}
Every finitely generated group $G$ can be coarsely separated by a set $Y$ of asymptotic dimension strictly less than $asdimG$ (provided that $asdimG>0$).
\end{thm}

The strategy of proof of theorem \ref{Coarse.1} is almost identical to that of theorem \ref{Coarse.3}. Thus we only prove theorem \ref{Coarse.3}. and theorem \ref{Coarse.2}.

We also give an example of a geodesic metric space which cannot be coarsely separated by any subset.\\

We note that for proper metric spaces the analogue of Theorem \ref{Coarse.1} might follows from \cite{DZ}.
%Then the strategy of the proof theorem \ref{Coarse.3} can be used to show that every graph $\Gamma$ containing a bi-infinite geodesic contains subsgraph $Y$ such that $asdimY < asdim \Gamma$ and which coarsely separates $\Gamma$. Applying 
%corollary 21 from \cite{BD08} we obtain theorem \ref{Coarse.1}.
We recall that a topological space $X$ of dimension $n$ is a \textit{Cantor $n$-manifold} if $X$ cannot be separated by a closed $(n-2)$-dimensional subset ($n \geq 2$). We give the following asymptotic analogue of Cantor manifolds:

\begin{defn}
Let $X$ be a connected metric space with $asdimX =n$, where $n \geq 2$. We say that $X$ is an \textit{$n$-asymptotic  Cantor manifold} if $X$ cannot be coarsely separated by a $k$-asymptotic dimensional subset, where $k \leq n\!-\!2$.
\end{defn}

In 1947 (see \cite{Aleksandrov}) P.S. Aleksandrov proved the following theorem:

\begin{thm} \label{Coarse.4}
For $n \geq2$, every $n$-dimensional Hausdorff topological space $X$ contains an $n$-dimensional Cantor manifold. 
\end{thm}

We prove two asymptotic analogues of this theorem, and we give some trivial examples of groups that cannot be asymptotic Cantor manifolds. In particular we prove:

\begin{thm} \label{Coarse.5}
Every finitely generated group of asymptotic dimension $n$, where $n \geq 2$, contains an asymptotic Cantor manifold. 
\end{thm}

In fact, we prove a more general result in the section 4.\\

The paper is organized as follows. In section 2, we start with some basic definitions, we develop some important techniques, and we prove theorem \ref{Coarse.2}. Section 3 contains the proof of theorem \ref{Coarse.1}. In section 4, we prove two versions of Aleksandrov's theorem. In the last section, we state some open problems.
\\

\textbf{Acknowledgments:} I would like to thank Panos Papasoglu for his valuable advice during the development of this research work. I would also like to offer my special thanks to David Hume and Dawid Kielak.

\section{Weak Coarse Separation of Spaces.}
The asymptotic dimension $asdimX$ of a metric space $X$ is defined as follows: $asdimX \leq n$ if and only if for every $R > 0$ there exists a uniformly bounded covering $\mathcal{U}$ of $X$ such that the R-multiplicity of $\mathcal{U}$ is smaller than or equal to $n+1$ (i.e. every R-ball in $X$ intersects at most $n+1$ elements of $\mathcal{U}$).\\
There are many equivalent ways to define the asymptotic dimension of a metric space. 
%It turns out that the asymptotic dimension of an infinite tree is $1$ and the asymptotic dimension of $\mathbb{E}^{n}$ is $n$.

Let $X$ be a metric space and $\mathcal{U}$ a covering of $X$, we  denote by $mesh(\mathcal{U})= sup \lbrace diam(U) \mid U \in \mathcal{U} \rbrace $. We say that the covering is \textit{$D$-uniformly bounded} (or $D$-bounded) if $mesh(\mathcal{U}) \leq D$. For $G=\langle S \rangle$ a finitely generated group we define by $S_{d}$ (or $\partial B_{d}(1_{G})$) the \textit{boundary} of $B_{d}(1_{G})$ in $Cay(G,S)$, namely, $S_{d}=  \lbrace g \in  G \mid \parallel\!g\!\parallel =d \rbrace$.

The following theorem shows that every finitely generated group $G$ can be weakly coarsely separated by a subset $Y$ of asymptotic dimension strictly less than $asdimG$.

%\subsection{The Theorem}  
\newtheorem{theorem}{Theorem}
\begin{thm}\label{Coasre.thm6.4}
Let $G$ be a finitely generated group with asdim$(G)=n$, $n \geq 1$. Then there exists a set $\Gamma$ which has asymptotic dimension strictly less than $n$, and weakly coarsely separates $G$. 
\end{thm}

We need two auxiliary lemmas.

\newtheorem{lemma}[theorem]{Lemma}

\begin{lem}\label{Coasre.lem6.5}
Let $G=\langle S \rangle$ be a finitely generated group with asdim$(G)=n$. Then for every $R \in \mathbb{N}\setminus{0}$ there exist $D_R$,$d_{R} > 0$ such that for every $d \geq d_{R}$ there exists a $D_{R}$-uniformly bounded family  $\overline{\mathcal{U}}_{d}$, such that:\\
(i)  $(R-1)$-mult($\overline{\mathcal{U}}_{d}) \leq n$.\\
(ii) The union $\Gamma_{R,d}=\bigcup \overline{\mathcal{U}}_{d}$  separates the group G into at least two connected components, one of which is unbounded and one of which is bounded.\\
(iii) One of the bounded components contains the ball with center $1_{G}$ and radius $\frac{d}{2}$ of $Cay(G,S)$ denoted by $B_{\frac{d}{2}}(1_{G})$.
\begin{proof}
 For $R>0$ there exists a uniformly bounded cover $\overline{\mathcal{U}}$ of G with
$R$-mult($\overline{\mathcal{U}}$) $=n+1$  and mesh($\overline{\mathcal{U}}$)$=D_{R}$ ($D_{R}>10$).\\
 We choose $d > d_{R} \geq 10^{9}(R+D_R +1)$. Let $X=Cay(G,S)$.\\
 \\
% Then there exist subfamily of $\overline{\mathcal{U}}$, denoted by $\overline{\mathcal{U}}_{S_{d}}$ that covers the sphere $S_{d}$.\\
We define $\overline{\mathcal{U}}_{S_{d}}= \lbrace U \in \overline{\mathcal{U}} \mid U \cap S_{d} \neq \varnothing \rbrace$
Let $U  \in \overline{\mathcal{U}}_{S_{d}}$ then we denote by

\begin{center} 
$U^{\ast} = \lbrace g\in U \cap G  $ such that there exists a path $ p:[0,l] \longrightarrow X $      from $ 1_{G} $ to $ g $ and $g$ is the first element of the path in $\bigcup \overline{\mathcal{U}}_{S_{d}} \rbrace$
\end{center}

\underline{Remark:} We may assume that the path $ p:[0,l] \longrightarrow X $ is simple. Obviously, we can omit loops in X.\\
\\
We say that $U^\ast$ is the \emph{descendant} of $U$.
%meets at most one time each element of the ball $B_{\Vert p(l) \Vert}(p(0))$.
Let: \begin{equation}
\overline{\mathcal{U}}_{d}= \lbrace U^{\ast} \mid U \in \overline{\mathcal{U}}_{S_{d}} \rbrace.
\end{equation}
We observe that mesh($\overline{\mathcal{U}}_{d})\leq D_{R}$. \\
\begin{figure}
\centering
\begin{subfigure}{.5\textwidth}
  \centering
  \includegraphics[width=.99\linewidth]{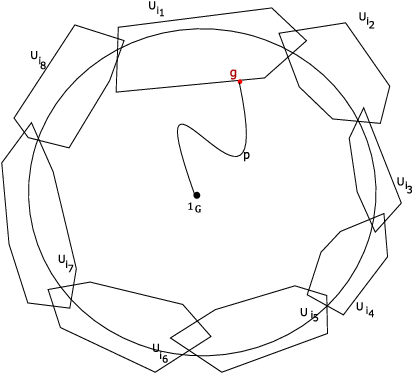}
  \caption{$\overline{\mathcal{U}}_{S_d }$}
  \label{fig:sub1}
\end{subfigure}%
\begin{subfigure}{.5\textwidth}
  \centering
  \includegraphics[width=.99\linewidth]{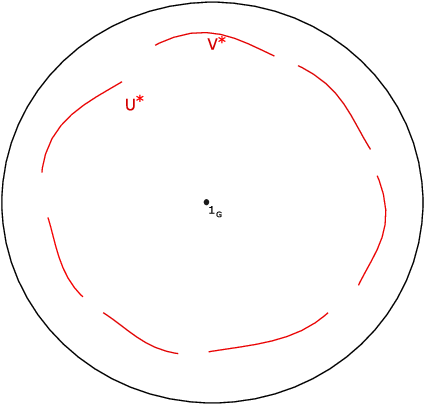}
  \caption{Here we have the union $\bigcup \overline{\mathcal{U}}_{d}$ }
  \label{fig:sub2}
\end{subfigure}
\caption{}
\label{fig:test}
\end{figure}

\textbf{Claim 1}: $X \setminus \bigcup \overline{\mathcal{U}}_{d}$ has at least two connected components, one of which is bounded and one of which is unbounded. The bounded component contains the ball $B_{d-(D_{R}+2)}(1_{G}) $.\\
Proof: We remark that $\bigcup \overline{\mathcal{U}}_{d} \subseteq B_{d+D_{R}}$ where $ B_{d+D_{R}}= B_{d+D_{R}}(1_{G})$, so $X \setminus B_{d+D_{R}}$ is a subset of $X \setminus \bigcup \overline{\mathcal{U}}_{d}$. Now it is easy to see that $X \setminus \bigcup \overline{\mathcal{U}}_{d}$ contains $B_{d-(D_{R}+2)}(1_{G}) $ and at least one unbounded connected subset $C$ of $X$.\\
We show that  $B_{d-(D_{R}+2)}(1_{G}) $ and  $C$ are disconnected in $X \setminus \bigcup \overline{\mathcal{U}}_{d}$. It suffices to show that every path $p:[0,l] \longrightarrow X$  from $1_{G}$ to $c \in C$ meets $\bigcup \overline{\mathcal{U}}_{d}$. \\
The path $p$ meets the sphere $S_{d}$, so there exists $U \in \overline{\mathcal{U}}_{S_{d}}$ in which the path $p$ has his first element in $\bigcup \overline{\mathcal{U}}_{S_{d}}$. Let's denote that element by $p_{0} \in U.$ Obviously $p_{0} \in U^{\ast}.$

% at least one unbounded connected subset of $X \diagdown \bigcup \overline{\mathcal{U}}_{d}}$. Lets denote by 
%\begin{center}
 %$C(\infty)= X \diagdown B_{d+D_{R}}}$
%\end{center}\\

%It is obvious that $B_{\frac{d}{2}}(1_{G}) \subseteq B_{d-(D_{R}+2)}}(1_{G}) \subseteq \overline{\mathcal{U}}_{d}}$ so to complete the proof of claim 1 it is enough to show that there is no path from $1_{G}$ to any element of $C(\infty).$\\
%Let $1_{G} \in B_{d-(D_{R}+2)}}$ , $c_{0} \in C(\infty)$ and $p:[0,l] \longrightarrow X$ to be a path from $1_{G}$ to $c_{0}$. The path $p$ goes through the sphere $S_{d}$, for this reason there exist $\mathcal{U} \in \overline{\mathcal{U}}_{S_{d}}$ in which the path $p$ has his first element in $\bigcup \overline{\mathcal{U}}_{S_{d}}$, let's denote that element by $p_{0}$.\\

\begin{figure}[h]
\begin{center}
\includegraphics{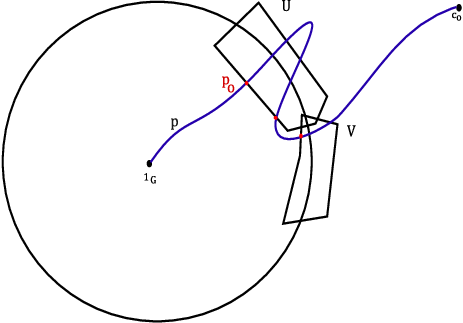}
\end{center}
%\vspace{0.1in}
 \caption{For $\mathcal{U},V \in \overline{\mathcal{U}}_{S_{d}}$}

\end{figure}

%By the definition of $\mathcal{U}^{\ast}$ we have that $p_{0} \in \mathcal{U}^{\ast}$ and the path $p$ goes through  $\bigcup \overline{\mathcal{U}}_{d}$, that proves the claim 1.\\

\textbf{Claim 2}: $(R-1)$-$mult(\overline{\mathcal{U}}_{d}) \leq n$.\\
Proof: Suppose that claim 2 does not hold. Then there exist $U^{\ast}_{1}$ , $ \ldots $ , $U^{\ast}_{n+1}$ elements in $\overline{\mathcal{U}}_{d}$ and a ball $B_{R-1}(x)$ which intersects all these elements.\\
If $u \in U^{\ast}_{1}$ then by definition there exists a path $\gamma:[0,l]\longrightarrow X$ from the identity to $u$ and its first element into $\bigcup \overline{\mathcal{U}}_{S_{d}}$ is $u$. Let's denote by $w$ the  element $\gamma(l-1) \in G$.\\
Then we have that $w \notin \bigcup \overline{\mathcal{U}}_{S_{d}}$ and $d_{S}(v,w)=1$, so
\begin{equation}
w \in B_{R}(x).
\end{equation}

\begin{figure}[h]
\begin{center}
\includegraphics{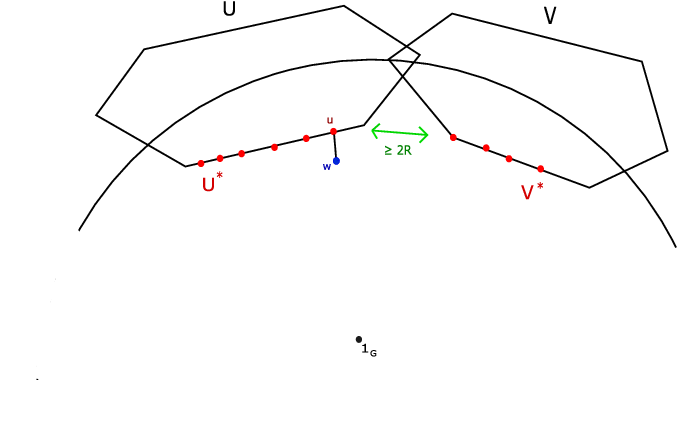}
\end{center}
%\vspace{0.1in}
 \caption{}
\end{figure}

In addition to, there exists $V \in \overline{\mathcal{U}} \setminus \overline{\mathcal{U}}_{S_{d}} $ such that $w \in V$,  so the ball $B_{R}(x)$ intersects $n+2$ discrete elements of $\overline{\mathcal{U}}$. This is a contradiction.

These two claims together prove the statement of the lemma.

\end{proof}

\end{lem}

%Now we have the last lemma,

\begin{lem}\label{Coasre.lem6.6}
Let $G=\langle S \rangle$ be a finitely generated group with asdim$(G)=n$. Let $R$, $D_R$, and $d \geq d_{R}$ be as the lemma \ref{Coasre.lem6.5}.\\
Assume $E$ is a bounded subset of  $Cay(G,S)$ which separates $Cay(G,S)$ into at least two connected components, one of which is bounded and one of which is unbounded. Assume further that one bounded connected component of $Cay(G,S)\setminus E$ contains the ball $B_{\frac{d}{2}}$ and 
\begin{equation}
E \subseteq B(\frac{3d}{2}) \setminus B(\frac{3d}{4}+1).
\end{equation}

Then for every $R^{\prime} \leq R$ there exists a $D_{R^{\prime}}$-uniformly bounded family $\overline{\mathcal{U}}_{R^{\prime}}$  such that:\\ 
(i) The $D_{R^{\prime}}$-neighbourhood of $E$ contains the union $\bigcup \overline{\mathcal{U}}_{R^{\prime}} .$\\
(ii) $(R^{\prime}-1)$-mult$(\overline{\mathcal{U}}_{R^{\prime}}) \leq n$. \\
(iii) The union $\bigcup \overline{\mathcal{U}}_{R^{\prime}} $ separates the $Cay(G,S)$ into at least two connected components, one of which is bounded and one of which is unbounded.\\
(iv) One of the bounded components contains the ball $B_{\frac{d}{2}}$.

\begin{proof}
The proof is exactly the same as the proof of the previous lemma. We do the same process by replacing the sphere $S_{d}$ by $E$. The only difference between them is the proof of claim 1, thus we will prove only the analogue of claim 1. We have the same symbolism as that of proof of lemma \ref{Coasre.lem6.5} except from $ \overline{\mathcal{U}}_{S_{d}}$, the analogue of $ \overline{\mathcal{U}}_{S_{d}}$ will be denoted by $\overline{\mathcal{U}}_{E}$. We also replace of $R$ by $R^{\prime}$. The family $ \overline{\mathcal{U}}_{d} $ consists of the \emph{descendants} of $\overline{\mathcal{U}}_{E}$. We can assume that $D_{R^{\prime}}=D_{R}=D$.\\

\textbf{Claim 1}: $X \setminus \bigcup \overline{\mathcal{U}}_{d}$ contains at least two connected components, one of which is bounded and one of which is unbounded. The bounded contains the ball $B_{\frac{d}{2}}$.

 Proof: We remark that  $\bigcup \overline{\mathcal{U}}_{d} \subseteq B(\frac{3d}{2}+D) \setminus B(\frac{3d}{4}-D)$, so $X \setminus \bigcup \overline{\mathcal{U}}_{d} $ contains the ball $B_{\frac{d}{2}}$ and at least one unbounded connected subset $C$ of $X$.\\
We have that the bounded set $E$ separates $Cay(G,S)$ into at least two connected components, one of which is bounded and one of which is unbounded. Thus, there exists an unbounded connected subset $C^{\prime}$ of $C$ such that $C^{\prime}$ is contained in an unbounded connected component $F$ of $X \setminus E$. 
%We have $C^{\prime} \subseteq F$.\\

We show that  $B_{\frac{d}{2}}$ and  $C^{\prime}$ are  disconnected in $X \setminus \bigcup \overline{\mathcal{U}}_{d}$. It suffices to show that every path $p:[0,l] \longrightarrow X$  from $1_{G}$ to $c \in C^{\prime}$ meets $\bigcup \overline{\mathcal{U}}_{d}$. \\
Since the subset $E$ separates $B(\frac{d}{2})$ and $F$ the path $p$ meets $E$, there exists $U \in \overline{\mathcal{U}}_{E}$ in which the path $p$ has it's first element in $\bigcup \overline{\mathcal{U}}_{E}$. Let's denote that element by $p_{0} \in U.$ Obviously $p_{0} \in U^{\ast}.$ This proves claim 1.

\end{proof}

\end{lem}
\begin{flushleft}
\textbf{Convention:} When we have a star ($\ast$) as superscript to a set $U$, namely $U^\ast$, then $U$ is related to the proofs of the previous two lemmas and $U^\ast$ is a \emph{descendant} of $U$. 
\end{flushleft}

%We move to the proof of theorem \ref{Coasre.thm6.4}.
%\subsection{The proof of the theorem}
%Now lets move on the proof of the theorem but first we give a sketch of th proof.% \\
%Now we have lemma2 and lemma4% 
\begin{proof}(\emph{Theorem \ref{Coasre.thm6.4}})

We follow some steps in order to construct the wanted sequence of subsets $\lbrace \Gamma_{k}\rbrace_{k \in \mathbb{N}} $ of $G$. Let $S$ be a finite generating set of $G$.\\
\\
\textbf{CONSTRUCTION I:}\\
 
\textbf{Step 1}: We choose $R_{1} + 1>100$. By lemma \ref{Coasre.lem6.5} there exist $ d_{R_{1}} > 0$ and a $D_{R_{1}}$-uniformly bounded family $\overline{\mathcal{U}}_{d_{R_{1}}}$ with  $R_{1}$-mult$(\overline{\mathcal{U}}_{d_{R_{1}}})\leq n$.
We take $d_{R_{1}} \geq 10^{9}(R_{1}+2D_{R_{1}}+1)$. The family $\overline{\mathcal{U}}_{d_{R_{1}}}$ was obtained from a $D_{R_{1}}$-uniformly bounded covering $\overline{\mathcal{U}}_{1}$ of $G$. 
We set $\Gamma_{0}^{1}= \bigcup \overline{\mathcal{U}}_{d_{R_{1}}}.$\\

\textbf{Step 2}: For $R_{2}+1 >10^{9}(R_{1}+2D_{R_{1}} +1)$ we take 
\begin{center}
$d_{R_{2}} \geq 10^{9}(d_{R_{1}}+R_{2}+2(D_{R_{2}}+D_{R_{1}})+2)$.
\end{center} 
By lemma \ref{Coasre.lem6.5} there exists a $D_{R_{2}}$-uniformly bounded family $\overline{\mathcal{U}}_{d_{R_{2}}}^{1}$ with $R_{2}$-mult$(\overline{\mathcal{U}}_{d_{R_{2}}}^{1})\leq n$.  The family $\overline{\mathcal{U}}_{d_{R_{2}}}^{1}$ was obtained from a $D_{R_{2}}$-uniformly bounded covering $\overline{\mathcal{U}}_{2}$ of $G$. We set $\Gamma_{R_{2}}^{1}= \bigcup \overline{\mathcal{U}}_{d_{R_{2}}}^{1}.$
\\
The conditions of lemma \ref{Coasre.lem6.6} are satisfied for $\Gamma_{R_{2}}^{1}$. Applying the proof of that lemma using the cover $\overline{\mathcal{U}}_{1}$ we obtain a $D_{R_{1}}$-uniformly bounded family
%  $\overline{\mathcal{U}}^{2}$  of $\Gamma_{R_{2}}^{1}$ that give us a refinement
$\overline{\mathcal{U}}_{d_{R_{2}}}^{2}$ with $R_{1}$-mult$(\overline{\mathcal{U}}_{d_{R_{2}}}^{2})\leq n$. We set $\Gamma_{R_{2}}^{2}= \bigcup \overline{\mathcal{U}}_{d_{R_{2}}}^{2}.$\\

\begin{flushleft}
To every $U_{i}^{1 \ast}$ in $\overline{\mathcal{U}}_{d_{R_{2}}}^{1}$ we associate a family, 
\begin{equation}
U_{i}^{1 \ast} \longrightarrow \lbrace U_{(i,s)}^{2 \ast}  , s \in I_{i}^1\rbrace.
\end{equation}

This family consists of all $U_{(i,s)}^{2 \ast} \in \overline{\mathcal{U}}_{d_{R_{2}}}^{2}$ such that $ U_{(i,s)}^{2} \cap U_{i}^{1 \ast} \neq \emptyset.$ (The set $U_{(i,s)}^{2 \ast}$ is a descendant of $U_{(i,s)}^{2 }$, and $U_{(i,s)}^{2 }$ is obtained by procedures of the previous lemmas.)
\end{flushleft}
%$U_{i,s}^2 = \lbrace g \in U_{s}^{2 \ast} \mid $ for $ U_{s}^{2 \ast} \in \overline{\mathcal{U}}_{d_{R_{2}}}^{2} $ and $ U_{s}^{2} \cap U_{i}^{1 \ast} \neq \emptyset \rbrace.$
We observe that:\\
(i) $\bigcup \overline{\mathcal{U}}_{d_{R_{2}}}^{2} = \bigcup \lbrace U_{(i,s)}^{2 \ast} \forall i , s \rbrace$\\
(ii) $R_{1}$-mult$(\lbrace U_{(i,s)}^{2 \ast} \forall i , s \rbrace)\leq n$\\
(iii) For $U_{i_{1}}^{1 \ast} $ , $\ldots $  , $  U_{i_{n+1}}^{1 \ast} \in \overline{\mathcal{U}}_{d_{R_{2}}}^{1}$ the sets $ \bigcup \lbrace U_{(i_{1},s)}^{2 \ast}  , s \in I_{i_{1}}^2 \rbrace$ , $\ldots $  , $ \bigcup \lbrace U_{(i_{n+1},s)}^{2 \ast}  , s \in I_{i_{n+1}}^2 \rbrace$ can not be intersected by a $(R_{2}-D_{R_{1}})$-ball.\\
(iv) We set $ \overline{\mathcal{U}}_{2}^{2} = \lbrace U_{(i,s)}^{2 \ast} \forall i , s \rbrace $, then
\begin{equation}
mesh(\overline{\mathcal{U}}_{2}^{2}) \leq D_{R_{2}} +2D_{R_{1}}.
\end{equation}
We set $\Gamma_{0}^{2}= \bigcup \overline{\mathcal{U}}_{2}^{2}$
.\\
.\\
.\\
.\\
\begin{figure}[h]
\begin{center}
\includegraphics{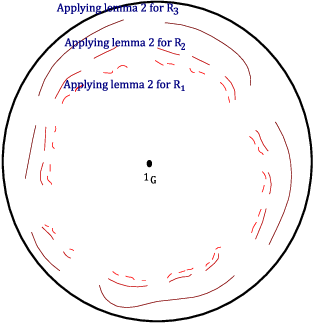}
\end{center}
\vspace{0.1in}
 \caption{Here we have an illustration of the result of step 3 for n=1.}
 \end{figure}
\textbf{Step n}: We take:
\begin{equation}
R_{n} \geq 10^{9}(R_{n-1}+2 \sum_{s=1}^{n-1}D_{R_{s}}+n )
\end{equation}
and
\begin{equation}
d_{R_{n}} \geq 10^{9}(d_{R_{n-1}}+R_{n}+2 \sum_{s=1}^{n-1}D_{R_{s}}+n).
\end{equation}
By lemma \ref{Coasre.lem6.5} there exists a $D_{R_{n}}$-uniformly bounded family $\overline{\mathcal{U}}_{d_{R_{n}}}^{1}$ with $R_{n}$-mult($\overline{\mathcal{U}}_{d_{R_{n}}}^{1}) \leq n$. We set $\Gamma_{R_{n}}^{1}= \bigcup \overline{\mathcal{U}}_{d_{R_{n}}}^{1}.$
\\
The conditions of lemma \ref{Coasre.lem6.6} are satisfied for $\Gamma_{R_{n}}^{1}$. Applying the proof of that lemma using the cover $\overline{\mathcal{U}}_{n-1}$ we obtain a $D_{R_{n-1}}$-uniformly bounded family
%covering  $\overline{\mathcal{U}}^{2}$  of $\Gamma_{R_{n}}^{1}$ that give us a refinement 
$\overline{\mathcal{U}}_{d_{R_{n}}}^{2}$ with $R_{n-1}$-mult($\overline{\mathcal{U}}_{d_{R_{n}}}^{2}) \leq n$. By $\overline{\mathcal{U}}_{n}^{2}$ we denote the analogue of $ \overline{\mathcal{U}}_{2}^{2}$ of the step 2.  We set $\Gamma_{R_{n}}^{2}= \bigcup \overline{\mathcal{U}}_{n}^{2}$.

We continue by applying the lemma \ref{Coasre.lem6.6} (n-2) more times and we end up having a $D_{R_{1}}$-uniformly bounded family such that $R_{1}$-mult($\overline{\mathcal{U}}^{n}_{d_{R_{n}}}) \leq n$.
\\
%\textbf{(i)} A set $\Gamma_{R_{n}}^{n}$\\
%\textbf{(i)} A $D_{R_{n}}$-uniformly bounded family 
%covering of $\Gamma_{R_{n}}^{n}$, $\overline{\mathcal{U}}_{n}^{n}$ 
%with $R_{1}$-mult($\overline{\mathcal{U}}_{n}^{n}) \leq n$.\\
%\textbf{(ii)} 
For every $k \in  \lbrace 1,\ldots ,n \rbrace$ and every fixed $ (i_{1},\ldots, i_{k}) $ we set 
\begin{center}
$U^{n, (i_{1},\ldots, i_{k}) }$ =  $\bigcup_{(i_{k+1},\ldots ,i_{n})}\lbrace U_{(i_{1},\ldots, i_{n})}^{n \ast},  $ and $ i^{}_{s} \in I_{(i_{1},\ldots, i_{s-1})}^n $ for every $ s \in \lbrace k+1,\ldots n \rbrace  \rbrace$.\\
\end{center}
We set $D_s = D_{R_{s}}$. We define the families:
\begin{center}
$\mathcal{V}^{n}_{k} = \lbrace U^{n,(i_{1},\ldots ,i_{k}) }  $ for every possible $ (i_{1},\ldots ,i_{k}) \rbrace$
\end{center}
Observe that,
\begin{center}
$mesh(\mathcal{V}^{n}_{k}) \leq D_{n-k+1} + 2\sum_{m=1}^{n-k}D_{m}. $
\end{center}
and 
\begin{center}
$(R_{n-k+1}- \sum_{m=1}^{n-k} D_{m})$-mult($\mathcal{V}^{n}_{k}) \leq n. $
\end{center}
We set $\Gamma_{0}^{n}= \bigcup \overline{\mathcal{U}}_{n}^{n}$, where $\overline{\mathcal{U}}_{n}^{n }= \lbrace U^{n \ast}_{(i_{1},\ldots ,i_{n}) }  $ for all $ (i_{1},\ldots ,i_{n}) \rbrace$. Here the step n ends.\\

Using the previous steps we constructed a sequence:
 \begin{equation}
 \lbrace \Gamma_{0}^{n}= \bigcup \overline{\mathcal{U}}_{n}^{n} \rbrace_{n \in \mathbb{N}},
 \end{equation}
and for every step n and every $k \in  \lbrace 1,\ldots ,n \rbrace$ the families:
\begin{center}
$\mathcal{V}^{n}_{k} = \lbrace U^{n,(i_{1},\ldots ,i_{k}) }  $ for every possible $ (i_{1},\ldots ,i_{k}) \rbrace$.
\end{center}
\textbf{CONSTRUCTION II:}\\

Let $c: \mathbb{R} \rightarrow Cay(G,S)$ be a bi-infinite geodesic such that $c(0)=1_{G}$. We choose a sequence of elements of G onto the geodesic $c$. The first element is $x_{1}=1_{G}$, the second element is $x_{2}=c(10^{9}(d_{R_{1}}+d_{R_{2}} + 2))$, so $d_{S}(x_{1},x_{2})=10^{9}(d_{R_{1}}+d_{R_{2}}+2)$. The n-nth element is:\\
\begin{equation}
x_{n}=c(10^{9}(n + d_{R_{n}} +  d_{R_{n-1}} +  d_{S}(x_{1},x_{n-1})).
\end{equation}\\
and 
$d_{S}(x_{n},x_{n-1})=10^{9}(n + d_{R_{n}}+ d_{R_{n-1}} + d_{S}(x_{1},x_{n-1}))$.\\
We denote by $\Gamma_{n}$ 
%the analogue of  $\Gamma^{n}_{0}$  if we do the same construction with replacement of $d_{R_{n}}$-ball by the ball %$B_{G}(x_{n},d_{R_{n}})$. Namely 
the transportation of $\Gamma^{n}_{0}$ by $x_{n}$.

\begin{equation}
\Gamma_{n}= x_{n} \cdot \Gamma^{n}_{0}
\end{equation}
 Let $\Gamma= \bigcup_{k=1}^{\infty}\Gamma_{k}$. Obviously, $\Gamma$ weakly coarsely separates $Cay(G,S)$. It remains to show the following, 
%For the final part of the proof we have the following claim.\\
\\
\textbf{Claim:}  asdim$\Gamma \leq n-1$.\\
\begin{figure}[h]
\begin{center} 
\includegraphics{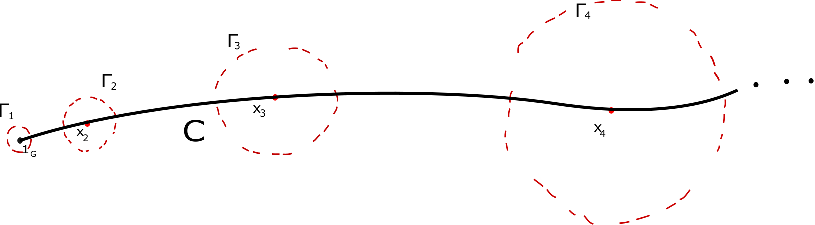}
\end{center}
%\vspace{0.1in}
 \caption{This is the picture we have.}

\end{figure}
Proof: Let $R>R_{1}$, then there exists a natural number $n_{0}$ such that $ n_{0}-1  > 10^{9}R.$

We observe that for every $k \in  \lbrace 1,\ldots ,n \rbrace$ each family $x_{n} \cdot \mathcal{V}^{n}_{k}$ is a $ (D_{n-k+1} +2\sum_{m=1}^{n-k}D_{m})$-uniformly bounded covering of $\Gamma_{n}$ and 
\begin{center}
 $(R_{n-k+1}- \sum_{m=1}^{n-k} D_{m})$-mult$(x_{n} \cdot \mathcal{V}^{n}_{k}) \leq n$.
\end{center} 
For every $n \geq 2 n_{0}+1$ we have that 
\begin{center}
$ R_{n-n_{0}}- \sum_{m=1}^{n-n_{0}} D_{m} > n_{0} > R$
\end{center}
We set $V_{0}= \bigcup_{n=1}^{2 n_{0}+1}\Gamma_{n}$. We will construct a uniformly bounded covering $\overline{\mathcal{V}}$ of $\Gamma$ with R-multiplicity less or equal to $n$. We consider the covering:
\begin{center}
$\overline{\mathcal{V}}=\lbrace V_{0} \rbrace \bigcup_{n=2 (n_{0}+1)}^{\infty} \lbrace x_{n} \cdot \mathcal{V}^{n}_{n- (2n_{0}+1)} \rbrace. $
\end{center}
Every element $V \in \overline{\mathcal{V}}$ has $diam(V)\leq max\lbrace diamV_{0} ,(D_{2 (n_{0}+1)} +2\sum_{m=1}^{2n_{0}+1}D_{m}) \rbrace$, so the cover $\overline{\mathcal{V}}$ is indeed uniformly bounded.

To complete the proof we need to show that every $R$-ball meets at most $n$ element of the covering.\\
\textbf{First}, we have that $V_{0}$ and $\overline{\mathcal{V}} \smallsetminus \lbrace V_{0} \rbrace$ are $(d_{S}(x_{2 (n_{0}+1)},x_{2 n_{0}+1})- (d_{R_{2 (n_{0}+1)}}+d_{R_{2 n_{0}+1}}))$-disjoint. Observe that:
\begin{center}
$ d_{S}(x_{2 (n_{0}+1)},x_{2 n_{0}+1})- (d_{R_{2 (n_{0}+1)}}+d_{R_{2 n_{0}+1}})  > n_{0} > 10^{9} R $,
\end{center}
 so there is no $R$-ball that meets both $V_{0}$ and an element of  $\overline{\mathcal{V}} \smallsetminus \lbrace V_{0} \rbrace$.\\
\textbf{Second}, let $2n_{0}+1<n_{1}<n_{2}$. If $x_{n_{1}} \cdot V^{n,(i_{1},\ldots ,i_{2n_{0}+1}) } \in x_{n_{1}} \cdot \mathcal{V}^{n_{1}}_{n_{1}-(2n_{0}+1)}$ and $x_{n_{2}} \cdot V^{n,(i_{1},\ldots ,i_{2n_{0}+1}) } \in x_{n_{2}} \cdot \mathcal{V}^{n_{2}}_{n_{2}-(2n_{0}+1)}$, then these two sets are $(d_{S}(x_{n_{1}},x_{n_{2}})-(d_{R_{n_{1}}}+d_{R_{n_{2}}}))-$disjoint. We have:\\
\begin{equation}
d_{S}(x_{n_{1}},x_{n_{2}}) - (d_{R_{n_{1}}}+d_{R_{n_{2}}})\geq d_{S}(x_{n_{2}-1},x_{n_{2}}) - (d_{R_{n_{2}-1}}+d_{R_{n_{2}}}) >  n_{2} > 10^{9}R,
\end{equation}
 so there is no $R$-ball that meets both those elements.\\
 %$V^{n,(i_{1},\ldots ,i_{n_{0}}) } \in x_{n_{1}} \cdot \mathcal{V}^{n_{1}}_{n_{1}-n_{0}}$ and $V^{n,(i_{1},\ldots ,i_{n_{0}}) } \in x_{n_{2}} \cdot \mathcal{V}^{n_{2}}_{n_{2}-n_{0}}$.\\
\textbf{Finally}, let $2n_{0}+1<n$. 
%We take $x_{n} \cdot \mathcal{V}^{n}_{n-n_{0}}$. 
We have that  $(R_{2(n_{0}+1)}-\sum_{m=1}^{2n_{0}+1}D_{m})$-mult($x_{n} \cdot \mathcal{V}^{n}_{n-(2n_{0}+1)}) \leq n$.
Since
\begin{equation}
R_{2(n_{0}+1)}-\sum_{m=1}^{2n_{0}+1}D_{m} > n_{0} >10^{9}R.
\end{equation}\\
we conclude that $R$-mult($x_{n} \cdot \mathcal{V}^{n}_{n-(2n_{0}+1)}) \leq n$.

\end{proof}

The proof of the next theorem is almost identical to that of theorem \ref{Coasre.thm6.4}.

\begin{thm}\label{Coarse.thm6.7}
Let $X$ be a geodesic metric space containing a bi-infinite geodesic. If $asdimX=n>0$, then there exists a subspace $Y$ of asymptotic dimension strictly less than $n$ which weakly coarsely separates $X$. 
\end{thm}

The following example shows that theorem \ref{Coarse.thm6.7} cannot be generalized to any geodesic metric space.
\begin{exmp}\label{coarse.ex2} Let $L$ be the infinite line $[0, \infty)$ with the graph metric. The asymptotic dimension of $L$ is one. We can observe that if $Y$ weakly coarsely separates $L$, then all the deep components of $L \setminus Y$ intersect each other since one of the deep connected components must be an unbounded interval.    
\end{exmp}	

%It is not hard to see that this example can generalized to any asymptotic dimension. 

\section{Coarse Separation of Spaces.}

\begin{thm}\label{Coarse.thm6.8}
Every finitely generated group $G$ can be coarsely separated by a set $Y$ of asymptotic dimension strictly less than $asdimG$ (provided that $asdimG>0$).
\end{thm}
\begin{proof} Let $\lbrace \Gamma_n \rbrace_n$ be the sequence of constructed in the proof of theorem \ref{Coasre.thm6.4} and let $c :\mathbb{R} \longrightarrow Cay(G,S) $ be the bi-infinite geodesic from Construction II of the proof of theorem \ref{Coasre.thm6.4}.\\
 We observe that the subsets $\Gamma_n$ look like spheres of $G$. The idea is to construct another sequence $\lbrace T_n \rbrace_n$ of subsets of $G$ using the same strategy but instead of starting from spheres of $G$ we start from ``tubes-cylinders'' of $G$. Then by combining these two sequences we construct a subset $Y$ of asymptotic dimension strictly less than $asdimG$ (provided that $asdimG>0$) which coarsely separates $G$.\\
\textbf{Discription of the sequence $\lbrace T_n \rbrace_n$:}\\
We show how to construct $T_1$, then using the same strategy as that of $\Gamma_n$ one can construct the rest of the sequence.

We consider the sequence of numbers $ \lbrace R_n \rbrace _n$ as defined in the proof of theorem \ref{Coasre.thm6.4}. We start with a large enough neighborhood of a long enough geodesic segment $c_1$ of $c$. Then using the same idea as that of lemma \ref{Coasre.lem6.5} and lemma \ref{Coasre.lem6.6} we obtain a $D_{R_{1}}$-uniformly bounded family $\overline{\mathcal{W}}_{d_{R_{1}}^{\prime}}$ with  $R_{1}$-mult$(\overline{\mathcal{W}}_{d_{R_{1}}^{\prime}}) \leq  asdim\,G$.
Here the role of $d_{R_{1}}$ will be played by $d_{R_{1}}^{\prime} = \dfrac{d_{R_{1}}}{2^7 }$. The family $\overline{\mathcal{W}}_{d_{R_{1}}^{\prime}}$ was obtained from the same $D_{R_{1}}$-uniformly bounded covering $\overline{\mathcal{U}}_{1}$ of $G$ as that used in the proof of theorem \ref{Coasre.thm6.4}. 
 
We set $T_{1}= \bigcup \overline{\mathcal{W}}_{d_{R_{1}}^{\prime}}.$ The set $T_1$ looks like a ``tube-cylinder'' of $G$. 
Observe that $T_1$ separates the $Cay(G,S)$ into at least two connected components, one of which is bounded and one of which is unbounded. One of the bounded components, say $\Sigma_1$ contains a large enough neighborhood of a long enough geodesic segment $c_1$ of $c([0, \infty))$.\\
% We observe that $asdim\, \cup T_n < asdim\,G $.\\ 
\textbf{Construction of $Y$:}\\
We define $Y$ in steps. We may choose the end points of the segments $c_n$ to be $x_n$ and $x_{n+1}$ as these have been defined in the proof of theorem \ref{Coasre.thm6.4}. We will describe only the step 1 since every other step n is the analogue of step 1.\\
\textbf{Step 1:}\\
\textbf{(a)} We remove the points of $T_1$ from the balls $B_{\dfrac{d_{R_{1}}}{2^3 }}(x_1)$ and $B_{\dfrac{d_{R_{2}}}{2^3 }}(x_2)$.\\ 
\textbf{(b)} We remove the points of $\Gamma_1$ and $\Gamma_2$ from $N_{\dfrac{d_{R_{1}}}{2^9 }}(c([0,\parallel  \!x_ {2}\! \parallel_S]))$. By $\parallel  \!x\! \parallel_S$ we denote the distance of $x$ from $1_G$, and by $N_r(A)$ we denote the $r$-neighborhood of $A$. Here the step 1 ends.

 We set $T_i^\prime =T_i \setminus (B_{\dfrac{d_{R_{i}}}{2^3 }}(x_i) \cup B_{\dfrac{d_{R_{i+1}}}{2^3 }}(x_{i+1}))$ and\\
  $\Gamma_i^\prime =\Gamma_i \setminus ( N_{\dfrac{d_{R_{i-1}}}{2^9 }}(c([\parallel  \!x_ {i-1}\! \parallel_S,\parallel  \!x_ {i}\! \parallel_S])) \cup  N_{\dfrac{d_{R_{i}}}{2^9 }}(c([\parallel  \!x_ {i}\! \parallel_S,\parallel  \!x_ {i+1}\! \parallel_S])))$.

Then we set $$Y=\cup_{i \in \mathbb{N}} (T_i^\prime \cup \Gamma_i^\prime)$$
 
We observe that $Cay(G,S) \setminus Y$ contains at least two deep components one of which contains $c :[0, \infty ) \longrightarrow Cay(G,S) $ and one of which contains $c :(- \infty , -N] \longrightarrow Cay(G,S) $ 
for some $N>0$. 

 By the construction of $Y$ it follows that $$Y \subseteq \cup_{i \in \mathbb{N}} (T_i^\prime \cup \Gamma_i).$$
Since $asdim \cup T_n^\prime < asdim\,G $ and $asdim\cup \Gamma_n < asdim\,G $, we conclude that: $$asdim\,Y< asdim\,G.$$
% we can see that the sequence $\lbrace T_n \rbrace_n$ satisfies the equality  $asdim\,\Gamma_n =0 $ uniformly. Moreover, 
%Then by applying the infinite union theorem (theorem 25 from \cite{BD08})
%We have that 
%Finally,  Indeed, it follows from the fact that , and 
\end{proof}

By the proof of the previous theorem we can generalize to the following:

\begin{thm}\label{Coarse.thm6.9}
Let $X$ be a geodesic metric space containing a bi-infinite geodesic. If $asdimX=n>0$, then there exists a subspace $Y$ of asymptotic dimension strictly less than $n$ which coarsely separates $X$. 
\end{thm}

The following example shows that theorem \ref{Coarse.thm6.9} cannot be generalized to any geodesic metric space.
\begin{exmp}\label{coarse.ex2} Let $A$ be the graph obtained by the wedge sum of infinitely many pointed intervals $\lbrace([0,a_n],0)\rbrace$, where $a_n \rightarrow \infty$. The asymptotic dimension of $A$ with the graph metric is one. We can observe that if $Y$ coarsely separates $A$, then all the deep connected components of $A \setminus Y$ must contain $0$, which is impossible.    
\end{exmp}	

Obviously, the previous example can generalized to any asymptotic dimension.

%\begin{thm}
%Every geodesic metric space of asymptotic dimension $n$, where $n \geq 2$, contains an asymptotic Cantor manifold. 
%\end{thm}

\section{Asymptotic analogues of Aleksandrov's theorem.}
We recall that a topological space X is a \textit{Cantor $n$-manifold}
if $X$ can not be separated by a closed $(n - 2)$-dimensional subset ($n \geq 2$). The concept of Cantor manifolds was introduced by P.S. Urysohn (see \cite{Ur}). Some examples of $n$-dimensional Cantor manifolds are the $n$-dimensional closed balls and the $n$-dimensional Euclidean space.\\

A maximal $n$-dimensional Cantor manifold in an $n$-dimensional compact space $X$ is called a \textit{dimensional component} of $X$. Each $n$-dimensional Cantor manifold of a compact Hausdorff space $X$  is contained in a unique dimensional component of $X$.\\
In 1947 (see \cite{Aleksandrov}) P.S. Aleksandrov proved the following theorem:

\begin{thm}
For $n \geq2$, every 
$n$-dimensional Hausdorff topological space $X$ contains an $n$-dimensional Cantor manifold. 
\end{thm}

It is natural to ask if there exists an analogue for the asymptotic dimension. We give the following definitions to make this precise. 

\begin{defn}
Let $X$ be a connected metric space with $asdimX =n$, where $n \geq 2$. We say that $X$ is an \textit{$n$-asymptotic Cantor manifold}  if $X$ cannot be coarsely separated by a $k$-asymptotic dimensional subset, where $k \leq n\!-\!2$.
\end{defn}

%A maximal $n$-dimensional Cantor manifold in an $n$-dimensional compact space $X$ is called a \textit{dimensional component} of $X$.

\begin{defn}

Let $X$ be a metric space with $asdimX =n$, where $n \geq 2$. We say that $Y \subseteq X$ is an \textit{asymptotic component} of $X$, if $Y$ is an $n$-asymptotic Cantor  manifold and it is maximal in the following sense: for every other $n$-asymptotic Cantor  manifold $A \subseteq X$ such that $Y \subseteq A$ there exists $R >0$ such that $A \subseteq N_{R}(Y)$.
\end{defn}
\begin{flushleft}
\textbf{Question 2.} Do asymptotic Cantor manifolds exist?
\end{flushleft}
The short answer to this question is, yes. For example, consider a collection of balls $B_r$ ($r \geq 0$) in the Cayley graph of $\mathbb{Z}^n$ (with respect to the usual generating set) with center at $(2^{r+2},0, \ldots ,0)$ and radius $r$. Then connect those balls by straight lines. The geodesic metric space obtained by this procedure (graph metric) is an $n$-asymptotic Cantor manifold.\\
One could alter question 2 and ask about asymptotic Cantor manifolds with specific properties.\\

We give two examples of groups that are not asymptotic  Cantor manifolds.

\begin{exmp}
Let $A,B$ be finitely generated groups with $max \lbrace asdimA, asdimB \rbrace = n$, and let $C$ a subgroup of them such that $asdimC \leq n-2$. Then the group $C$ coarsely separates the amalgamated product $ A \ast_{C} B $, and $ asdim A \ast_{C} B =n.$

\end{exmp}

\begin{exmp}
Let $G$ a finitely generated group with $asdimG = n$, and let $N$ be a subgroup such that $asdimN \leq n-2$. Then the group $N$ coarsely separates the HNN-extension $ G \ast_{N}  $, and $ asdim G \ast_{N} =  n.$

\end{exmp}

\begin{cor}
Let $G$ a finitely generated group with $asdimG=n$. If the group $G$ is an $n$-asymptotic Cantor manifold, then it does not split over a subgroup $H$ with $asdimH \leq n-2$.
\end{cor}

\subsection{Results.}

\begin{defn} Let $X$ be a metric space. We say that $X$ is \emph{point-transitive} if for any two points $x,y \in X$ there is an isometry, $f$, of $X$ such that $f(x)=y$.
\end{defn}

For example, any Cayley graph is point-transitive metric space.

\begin{defn} Let $X$ be a metric space. We say that $X$ has \emph{uniformly coarse connected spheres} if there exists $R>0$ such that the R-neighbourhood of the connected components of any sphere is connected.
\end{defn}

\begin{defn} Let $X$ be a metric space. We say that $\Sigma \subseteq X$ is \emph{c-net} of $X$, where $c\geq 0$, if any element of $X$ is at distance at most $c$ from $\Sigma$.
\end{defn}

\begin{defn} Let $X$ be a metric space. We say that $X$ is \emph{net-compact} if there exists a $c$-net, $\Sigma$, of $X$ such that every ball in $X$ contains finitely many elements of $\Sigma$.
\end{defn}

The following theorems are asymptotic versions of Aleksandrov's theorem.

\begin{thm}
Let $X$ be a point-transitive geodesic metric space of asymptotic dimension $n$, where $n \geq 2$. If $X$ is net-compact and has at least one infinite geodesic, then $X$ contains an asymptotic Cantor manifold. 
\end{thm}
\begin{proof} Let $\gamma$ be an infinite geodesic of $X$ issuing from $x_1$. We construct a set of balls $\lbrace B_n \rbrace$ as follows:\\
Let $B_1$ be the ball of radius $1$ and center $x_1$,
% $B_2$ be the ball of radius $2$ and center $x_2$, where $d(x_1, x_2)=2^2$, 
and $B_n$ be the ball of radius $n$ and center $x_n$, where $x_n \in \gamma$ and $d(x_1, x_n)=2^n$.\\

Let $\gamma_n$ be the geodesic segment of $\gamma$ which connects the boundary of the ball $B_{n-1}$ with the boundary of the ball $B_n$. We set $B=(\bigcup_n B_n )\cup \gamma$. We observe that $B$ is connected.\\

\begin{flushleft}
\textbf{Claim 1:} $asdim B=n$.
\end{flushleft}
\begin{proof}
Let $\Sigma$ be a $c$-net of $X$ such that every ball in $X$ contains finitely many elements of $\Sigma$. Observe that $\Sigma$ is quasi-isometric to $X$, so $asdim\Sigma= asdim X=n$.

We assume that $asdim B<n$, then for every $R > 0$ there exists a uniformly bounded covering $\mathcal{U}_R$ of $B$ such that the R-multiplicity of $\mathcal{U}_R$ is at most $n$.

We consider the coverings $\mathcal{U}_R^n = \lbrace U \in \mathcal{U}_R \vert U \cap B_n \neq \emptyset \rbrace$ of the balls $B_n$. Since $X$ is point-transitive we can ''translate''  $B_n$'s to $x_1$. Then we observe that the translations of the coverings $\mathcal{U}_R^n$, forms coverings for the balls of radius $n$ around $x_1$, for any $n$.\\
We further observe that since each ball $B_n(x_1)$ contains finitely many elements of $\Sigma$ we can progressively construct a uniformly bounded cover, $\mathcal{V}_R$, of $\Sigma$ using the translations of the coverings $\mathcal{U}_R^n$. By construction, the R-multiplicity of $\mathcal{V}_R$ is at most $n$, thus $asdim\Sigma < n$, which is a contradiction.
\end{proof}

\begin{flushleft}
\textbf{Claim 2:} $B$ is an asymptotic Cantor manifold.
\end{flushleft}
\begin{proof}
Let $S$ be a subset of $B$ which coarsely separates $B$. So, $X \setminus S$ contains at least two deep components $C_1$ and $C_2$.\\
Observe that $C_1$ must contain $\cup_{i \geq k} \gamma_i$, for some $k$ (since $C_1$ is connected).
Similarly, $C_2$ must contain $\cup_{i \geq m} \gamma_i$, for some $m$. Then $C_1$ and $C_2$ are not disjoint, which is a contradiction.

\end{proof}
This completes the proof of the theorem.

\end{proof}

\begin{cor}
Every finitely generated group of asymptotic dimension $n$, where $n \geq 2$, contains an asymptotic Cantor manifold. 
\end{cor}

\begin{thm}
Let  $X$ be a geodesic metric space of asymptotic dimension $n$, where $n \geq 2$. If $X$ contains at least one infinite geodesic and has uniformly coarse connected spheres, then $X$ contains an asymptotic Cantor manifold. 
\end{thm}
\begin{proof}
Let $\gamma$ be an infinite geodesic of $X$ issuing from $x_1$. We assume that $X$ has uniformly connected spheres with constant $c>1$. We consider annuli $A_n= \lbrace x \in X \vert 2^{c(n+1)}\leq d(x_1, x) \leq 2^{c(n+2)} \rbrace$. We set $X_0= \cup_{n \in \lbrace 2k \rbrace }A_n$, and $X_1= \cup_{n \in \lbrace 2k+1 \rbrace}A_n$.\\

%We have $X \setminus B_{2^{c(n+1)}}(x_1)= X_0 \cup X_1 $.
We recall:

\begin{thm}{(Finite Union Theorem, see \cite{BD01}})
For every metric space presented as a finite union $X= \cup_{i} X_{i}$ we have
\begin{center}
$asdimX = max\lbrace asdimX_{i} \rbrace$.
\end{center}
\end{thm}

In view of the finite union theorem and the fact that $X \setminus$ ball is quasi-isometric to $X$, we obtain that either $X_0$ or $X_1$ (or both) has asymptotic dimension equal to $n$.\\
Without loss of generality we assume that $asdimX_0=n$. We set $\Gamma = N_c(X_0) \cup \gamma$, observe that $\Gamma$ is connected and $asdim\Gamma=n$.\\

%2^{c(n+1)}\leq d(x_1, x) \leq 2^{c(n+2)} 

We will show that $\Gamma$ is an asymptotic Cantor manifold. Otherwise, let $S$ be a subset of $\Gamma$ which coarsely separates $\Gamma$. So, $\Gamma \setminus S$ contains at least two deep components $C_1$ and $C_2$. Let $y_k$ be the mid-point of the geodesic segment of $\gamma$ that connects the sphere (with center $x_1$) of radius $2^{c((2k-1)+1)}$ with the sphere (with center $x_1$) of radius $2^{c(2k+1)}$.

Observe that $C_1$ must contain $\cup_{i \geq m_1 } y_i$, for some $m_1$ (since $C_1$ is connected).
Similarly, $C_2$ must contain $\cup_{i \geq m_2} y_i$, for some $m_2$. Then $C_1$ and $C_2$ are not disjoint, which is a contradiction.

\end{proof}

%The following example shows that the assumptions of \emph{point-transitivity} and \emph{uniformly coarse connected spheres} are necessary.
%\begin{exmp}\label{coarse.ex2} Let $n$ be a natural number larger than $1$, and let $A$ be the graph obtained by the wedge sum of infinitely many pointed intervals $\lbrace([0,\infty)_k,0)\rbrace$. We fix a point $a_k$ inside each interval $[0,\infty)_k$, where $a_k<a_{k+1}$ and $a_k \rightarrow \infty$. For each $k$, we glue to $a_k$ a copy of a ball, $B_k$, of the Cayley graph of $\mathbb{Z}^n$ with radius $k$. Let $X_n$ be the resulting graph. We consider $X_n$ as a metric space with the graph metric.\\
%Observe that $X_n$ is a net-compact geodesic metric space that contains an infinite geodesic and has asymptotic dimension $n$.

%We claim that $X_n$ doesn't contain any $n$-asymptotic Cantor manifold. Indeed, if $\Gamma \subseteq X_n$ is an   
%$n$-asymptotic Cantor manifold, then $\Gamma$ must intersect infinitely many balls $B_k$, because $asdim \Gamma =n$. Also, since $\Gamma$ is connected and intersects infinitely many balls $B_k$, then it must contains the base point $0$. Observe that the   

%The asymptotic dimension of $A$ with the graph metric is one. We can observe that if $Y$ coarsely separates $A$, then all the deep connected components of $A \setminus Y$ must contain $0$, which is impossible.    
%\end{exmp}	

\section{Questions.}

\begin{flushleft}
\textbf{Question 3.} Do asymptotic components always exist?
\end{flushleft}

%We make the following

%\textbf{Question 2.} Does there exist an analogue of the theorem 3.1 for asymptotic Cantor manifolds ?\\

%\textbf{Question 3a.} Does there exist $n$-asymptotic dimensional Cantor manifold of a group $G$ that corresponds to a coset of a subgroup ?\\

\begin{flushleft}
\textbf{Question 4.} Under which circumstances asymptotic components of a group $G$ correspond to subgroups ?
\end{flushleft}

\begin{flushleft}
\textbf{Question 5.} Let $G$ be a finitely generated group, and let $S$ be a subset of $G$ such that $asdimS < asdim G$. We further assume that $S$ coarsely separates $G$. Under which circumstances $G$ splits over a subgroup quasi-isometric to $S$?
\end{flushleft}

Stallings showed that if a compact set coarsely separates the Cayley graph of a finitely generated group $G$, then G splits over a finite group. However, P. Papazoglou and T.Delzant (see \cite{PD}) showed that for any $n$ there is a hyperbolic group $G$ with $asdimG > n$ which is separated coarsely by a uniformly embedded set $S$ of $asdimS = 1$ and which has property FA, so it does not split.

\begin{conj}
Every $n$-dimensional Euclidean space is an $n$-asymptotic Cantor manifolds, where $n \geq 2$.
\end{conj}

\begin{conj}
Every $n$-dimensional hyperbolic space, $\mathbb{H}^n$, is an $n$-asymptotic Cantor manifolds, where $n \geq 2$.
\end{conj}

\begin{flushleft}
\textit{Address:}  École Normale Supérieure, 45 Rue d'Ulm, 75005 Paris\\
\emph{Email:} Panagiotis.Tselekidis@ens.fr
\end{flushleft}

%Mathematical Institute, University of Oxford, Andrew Wiles Building, Woodstock Rd, Oxford OX2 6GG, U.K.

%\textit{Address 2:} The Queen's College, High Street,
%Oxford, OX1 4AW, U.K.

\end{document}